\documentclass[12pt,reqno]{amsart}
\usepackage{latexsym,amssymb,amscd}
\usepackage{graphicx}
\def\kk{{\Bbbk}}
\def\Lc{{\mathcal{L}}}
\def\B'c{{\mathcal{B'}}}
\def\U'c{{\mathcal{U'}}}

\def\Hc{{\mathcal{H}}}
\def\Ec{{\mathcal{E}}}

\def\Fc{{\mathcal{F}}}

\def\xb{{\bold x}}

\def\opn#1#2{\def#1{\operatorname{#2}}} 
\opn\chara{char}
\opn\length{\ell}
\opn\projdim{proj\,dim}
\opn\injdim{inj\,dim}
\opn\ini{in}
\opn\rank{rank}
\opn\height{ht}
\opn\Tiefe{Tiefe}
\opn\grade{grade}
\opn\height{ht}
\opn\embdim{emb\,dim}
\opn\codim{codim}

\opn\Tr{Tr}
\opn\bigrank{big\,rank}
\opn\superheight{superheight}\opn\lcm{lcm}
\opn\trdeg{tr\,deg}%
\opn\reg{reg}
\opn\lreg{lreg}
\opn\deg{deg}
\opn\lcm{lcm}
\opn\set{set}
\opn\ara{ara}
\opn\bight{bight}

\opn\div{div}
\opn\Div{Div}
\opn\cl{cl}
\opn\Cl{Cl}
\opn\Spec{Spec}
\opn\Supp{Supp}
\opn\supp{supp}
\opn\Sing{Sing}
\opn\Ass{Ass}
\opn\Min{Min}

\opn\Ann{Ann}
\opn\Rad{Rad}
\opn\Soc{Soc}
\opn\Ker{Ker}
\opn\Coker{Coker}
\opn\Im{Im}
\opn\Hom{Hom}
\opn\Tor{Tor}
\opn\Ext{Ext}
\opn\End{End}
\opn\Aut{Aut}
\opn\id{id}

\opn\nat{nat}
\opn\GL{GL}
\opn\SL{SL}
\opn\mod{mod}
\opn\ord{ord}
\opn\depth{depth}
\opn\set{set}
\opn\Shad{Shad}
\opn\pd{pd}
\opn\aff{aff}
\opn\con{conv}
\opn\relint{relint}
\opn\st{st}
\opn\lk{lk}
\opn\cn{cn}
\opn\core{core}
\opn\vol{vol}
\opn\gr{gr}

\opn\d{d}
\opn\Ind{Ind}
\opn{\indmat}{indmat}

\newcommand{\diam}{\mbox{\rm{diam}}}
\newcommand{\syz}{\mbox{\rm{syz}}}

\def\pot#1#2{#1[\kern-0.28ex[#2]\kern-0.28ex]}

\opn\dirlim{\underrightarrow{\lim}}
\opn\invlim{\underleftarrow{\lim}}
\def\pnt{{\raise0.5mm\hbox{\large\bf.}}}

\def\Implies{\ifmmode\Longrightarrow \else
     \unskip${}\Longrightarrow{}$\ignorespaces\fi}
\def\implies{\ifmmode\Rightarrow \else
     \unskip${}\Rightarrow{}$\ignorespaces\fi}
\def\iff{\ifmmode\Longleftrightarrow \else
     \unskip${}\Longleftrightarrow{}$\ignorespaces\fi}

\let\:=\colon
\newtheorem{Theorem}{Theorem}[section]
\newtheorem{Lemma}[Theorem]{Lemma}
\newtheorem{Corollary}[Theorem]{Corollary}
\newtheorem{Proposition}[Theorem]{Proposition}
\newtheorem{Remark}[Theorem]{Remark}

\newtheorem{Example}[Theorem]{Example}

\newtheorem{Definition}[Theorem]{Definition}
\newtheorem{Problem}[Theorem]{Problem}

\let\epsilon=\varepsilon
\let\kappa=\varkappa
\title{Line Graphs of Simplicial Complexes}
\author{Anda Olteanu}
\address{Faculty of Marine Engineering, ``Mircea cel B\u atr\^an" Naval Academy, Fulgerului Street, no. 1,
	900218 Constanta, Romania,} \email{olteanuandageorgiana@gmail.com}
\begin{document}
	
	\maketitle
	\begin{abstract}  We consider the line graph of a pure simplicial complex. We prove that, as in the case of line graphs of simple graphs, one can compute the second graded Betti number of the facet ideal of a pure simplicial complex in terms of the combinatorial structure of its line graph. We characterize those pure simplicial complexes whose line graph is a complete (bipartite) graph. We give conditions that line graphs of simplicial complexes should fulfill.
	\end{abstract}

	\section*{Introduction}
	
	Defined by Whitney \cite{Wh} line graphs have been intensively studied in graph theory. This concept has been introduced under different names by many authors but the term ‘line graph’ was later introduced by Hoffmann \cite{Ho}. A characterization of those graphs which are line graphs of some graph can be found for instance in \cite{B} and \cite{HN}. Due to their properties, a lot of work has been done for generalizing them for hypergraphs, \cite{BHS,LT, TZ}. For instance, Tyshkevich and Zverovich defined the line graph of a hypergraph $\Hc$ as the graph whose vertex set is the edge set of $\Hc$, and two vertices are adjacent in $L(\Hc)$ if the corresponding edges are adjacent or intersecting edges in $\Hc$, \cite{TZ}. Moreover, in \cite{BHS} Bermond, Heydemann, and Sotteau considered the $k$-line graph of a hypergraph as being the graph with the vertex set given by the set of edges of the hypergraph and two vertices are adjacent if the intersection of the corresponding edges of $\Hc$ has at least $k$ elements.
	
	From commutative algebra point of view, line graphs of simple graphs appear in the computation of the second graded Betti number of its edge ideal as Eliahou and Villarreal proved \cite{EV}. We are mainly interested in simplicial complexes and their facet ideals, which can be viewed as edge ideals of hypergraphs. Therefore we consider pure simplicial complexes of dimension $d-1$, $d\geq2$, and we use the definition of $k$-line graphs given in \cite{BHS}, with $k=d-1$. Since this is the only graph that we consider, we will call it the line graph of the simplicial complex. Under this definition, line graphs appears also in the study of polytopes as \textit{facet-ridge incidence graphs}. For instance, Blind and Mani proved that simplicial polytopes are completely determined by their facet-ridge graphs \cite{BM}.
	
	For line graphs of simplicial complexes we prove several results which are similar to the case of simple graphs. We give a similar result for the second graded Betti number of the facet ideal of a pure simplicial complex $\Delta$ [Theorem~\ref{Bettisc}] and, as a consequence, we obtain a connection between the second Betti number of the Stanley-Reisner ideal of the Alexander dual, $S/I_{\Delta^{\vee}}$ and the line graph of $\Delta$, [Corollary ~\ref{Bettiscaldual}]. We consider the particular class of complete graphs $\mathcal{K}_n$, $n\geq3$, and we characterize all the pure simplicial complexes whose line graph is $\mathcal{K}_n$ [Proposition \ref{complete}]. Moreover we characterize the complete bipartite graphs which are line graphs of some simplicial complex [Theorem \ref{complete-bipartite-mn}]. We pay also attention to general properties of line graphs of simplicial complexes such as the connectivity, the number of edges, the behaviour on deletion of facets or addition of vertices. We also determine classes of forbidden subgraphs of the line graphs (Proposition \ref{1,d+1}, Theorem \ref{complete-bipartite-mn}).There is a complete characterizations of those graphs which are line graphs of some graph (Theorem \ref{LG}). Taking into account this result, we give a sufficient condition for line graphs of simplicial complexes (Theorem \ref{charlinesc}). Moreover, we construct classes of graphs for which it is also necessary.  
	
	The paper is structured in eight sections. In the first section we recall basic notions of simplicial complexes, hypergraphs, graphs and edge ideals. The second section is devoted to defining line graphs of pure simplicial complexes and studying combinatorial properties such as being connected or a formula for the number of edges are determined.
	The behaviour of line graphs and simplicial complexes on adding or deleting vertices/facets is studied.
	
	The third and fourth sections are devoted to particular classes of graphs such as cycles, complete graphs and complete bipartite graphs. For complete (bipartite graphs) we give full descriptions of the corresponding simplicial complexes. For the case of cycles, we only construct classes of simplicial complexes for which the line graph is a cycle Proposition \ref{cycle}, and we give full characterizations for $C_3$ and $C_4$ [Corollaries \ref{C3} and \ref{C4}].
	
	In the next section we pay attention to a particular class of graphs which appears as a forbidden class for intersection graphs of linear $d$-uniform hypergraphs. We show that there are pure simplicial complexes with the line graph from this class Theorem~\ref{G1t} 
	
	Section six is devoted to generalizing result which hold for line graphs of graphs [Theorem ~\ref{LG}]. We obtain a sufficient condition for line graphs of simplicial complexes [Theorem \ref{charlinesc}] and we prove that it is not necessary. nevertheless, we determine classes of graphs for which the condition is also necessary Theorems \ref{triangle-wheel} and \ref{wheel}.
	
	The next section is devoted to applications to the resolutions of facet ideals of pure simplicial complexes. More precisely, we prove Theorem~\ref{Bettisc} which is similar to the one given by Eliahou and Villarreal \cite[Proposition 2.1]{EV}. By using properties of line graphs, we also relate the Stanley-Reisner ideal of the Alexander dual to line graphs [Corollary \ref{Bettiscaldual}].
	
	The last section is devoted to final conclusions and remarks. We pointed out here some problems which arise from our computations and that are of interest in understanding line graphs of simplicial complexes.
	
	\section{Background}
	In this section we recall the notions and properties that will be used later. For more details, one may see \cite{BrHe,HaTu,HaTu1,HeHi,MV,Vi}.
	\subsection{Simplicial complexes}
	\textit{A simplicial complex} $\Delta$ on the vertex set $\{x_1,\ldots, x_n\}$, where $n\geq1$ is an integer, is a collection of subsets (called \textit{faces}) such that any vertex is in $\Delta$ and, if $F$ is a face of $ \Delta$ and $G\subset F$, then $G$ is also a face of $\Delta$. Maximal faces (with respect to the inclusion) are called \textit{facets} and their set is denoted by $\Fc(\Delta)$. Moreover, if $\Fc(\Delta)=\{F_1,\ldots,F_r\}$, then $\Delta=\langle F_1,\ldots, F_r\rangle$ is just another way to write the simplicial complex with facets $\Fc(\Delta)$.  \textit{A simplex} is a simplicial complex with only one facet. The \textit{dimension} of the simplicial complex is denoted by $\dim(\Delta)$ and is defined as $\dim(\Delta)=\max\{\#(F)-1:F\in\Delta\}$. A simplicial complex is \textit{pure} if all its facets have the same dimension and a $d$-simplex si a simplex of dimension $d$, that is $\langle\{x_1,\ldots,x_{d+1}\}\rangle$.
If $\Delta_1$ and $\Delta_2$ be simplicial complexes with disjoint sets of vertices $V_1$ and $V_2$, \textit{the join} of $\Delta_1$ and $\Delta_2$ is the simplicial complex $$\Delta_1*\Delta_2=\{F\cup G: F\in\Delta_1,\ G\in \Delta_2\}.$$
	
	\textit{The Alexander dual} of a simplicial complex $\Delta$, denoted by $\Delta^{\vee}$, is the simplicial complex with the faces given by the complementary of non-faces of $\Delta$, that is $$\Delta^{\vee}=\{F^c:F\notin\Delta\}.$$
	
	For a simplicial complex $\Delta$, let $\Delta^c$ be the simplicial complex with the facet set $$\Fc(\Delta^c)=\{F^c:F\in\Fc(\Delta)\}.$$
	
	A simplicial complex is called \textit{shellable} if there is an ordering of its facets $F_1,\ldots,F_r$ such that for all $i,j$ with $1\leq j<i\leq r$, there exist a vertex $x\in F_i\setminus F_j$ and an integer $k<i$ such that $F_i\setminus F_k=\{x\}$.

	\subsection{Squarefree monomial ideals associated to simplicial complexes}
	Let $\Delta$ be a simplicial complex on the vertex set $V=\{u_1,\ldots, u_n\}$ and $\kk$ a field. Let $S=\kk[x_1,\ldots,x_n]$ be the polynomial ring in $n$ variables over the field $\kk$. To a set $F=\{i_1,\ldots,i_t\}\subseteq V$, one may associate the squarefree monomial $\xb_{F}=x_{i_1}\cdots x_{i_t}\in S $ and we will refer to $F$ as\textit{ the support of the monomial} $\xb_F$. 
	
	For the simplicial complex $\Delta$ two squarefree monomial ideals are of interest:\begin{itemize}
		\item \textit{the Stanley--Reisner ideal} $I_{\Delta}$ which is generated by the squarefree monomials which correspond to the minimal non-faces of $\Delta$,
		$$I_\Delta=(\xb_F:F\notin\Delta)$$
		\item \textit{the facet ideal} $I(\Delta)$ which is generated by the squarefree monomials which correspond to the facets of $\Delta$,
		$$I(\Delta)=(\xb_F:F\in\Fc(\Delta))$$
	\end{itemize}
	We will write $\kk[\Delta]$ for \textit{the Stanley--Reisner ring of} $\Delta$, that is $\kk[\Delta]=S/I_{\Delta}$
	
	If we consider the Stanley--Reisner ideal of the Alexander dual of $\Delta$, then one has the following result:
	\begin{Proposition}\cite[Lemma 1.2]{HeHiZh}\label{Aldualcomp} If $\Delta$ is a simplicial complex, then
	$I_{\Delta^{\vee}}=I(\Delta^c)$ 
	\end{Proposition} 
		
	Let $I\subseteq S=\kk[x_1,\ldots,x_n]$ be an ideal and $\mathcal{F}$ is the minimal graded free resolution of $S/I$ as an $S$-module:
	\[\mathcal{F}: 0\rightarrow\bigoplus\limits_jS(-j)^{\beta_{pj}}\rightarrow\cdots\rightarrow\bigoplus\limits_j S(-j)^{\beta_{1j}}\rightarrow S\rightarrow S/I\rightarrow0,\] then the numbers $\beta_{ij}$ are \textit{the graded Betti numbers of $S/I$}.
	Let $d>0$ be an integer. An ideal $I$ of $S$ \textit{has a $d$--linear resolution} if the minimal graded free resolution of $S/I$ is of the form
	\[\ldots\longrightarrow S(-d-2)^{\beta_2}\longrightarrow S(-d-1)^{\beta_2}\longrightarrow S(-d)^{\beta_1}\longrightarrow S\longrightarrow S/I\longrightarrow 0.
	\]  If $d=2$, we simply say that the ideal has \textit{a linear resolution}.
	
	In between the combinatorics of simplicial complexes and the homological pro\-perties of the associate squarefree monomial ideals there are strong connections.
	\begin{Theorem}\rm(Eagon--Reiner)\cite{EaRe}\it$\ $ Let $\kk$ be a field and $\Delta$ a simplicial complex. Then $\kk[\Delta]$ is Cohen--Macaulay if and only if $I_{\Delta^{\vee}}$ has a linear resolution.
	\end{Theorem}
	We recall that a simplicial complex is Cohen--Macaulay if its Stanley--Reisner ring has this property.
	\begin{Definition}\rm\cite{HeTa}
		A monomial ideal $I$ of $S$ is called an \textit{ideal with linear quotients} \rm if there is an ordering of its minimal monomial set of generators $m_1,\ldots, m_r$ satisfying the following property: for all $\ 2\leq i\leq r$ and for all $j<i$, there exist $l$ and $k$, $l\in\{1,\ldots,n\}$ and $k<i$, such that $m_k/\gcd(m_k,m_i)=x_l$ and $x_l$ divides $m_j/\gcd(m_j,m_i)$. 
	\end{Definition}

	\begin{Theorem}\cite[Theorem 1.4]{HeHiZh}\label{shell}
	 Let $\kk$ be a field and $\Delta$ a pure simplicial complex. Then $\Delta$ is shellable if and only if $I_{\Delta^{\vee}}$ has linear quotients.
	\end{Theorem}
	\subsection{Graphs} Through this paper, all the graphs will be assumed to be simple (without loops or multiple edges). 
	Let $G$ be a finite simple graph with the vertex set $V(G)$ and the set of edges $E(G)$. Two vertices $u,v\in V(G)$ are called \textit{adjacent} if they form an edge in $G$. The neighbourhood of a vertex $u$ of $G$ is denoted by $\mathcal{N}(u)$ and is the set of all the neighbours of $u$, that is $\mathcal{N}(u)=\{v\in V(G)\,:\,\{u,v\}\in E(G)\}$. \textit{The closed neighbourhood of $u$}, denoted by $\mathcal{N}[u]$, is defined as $\mathcal{N}[u]=\mathcal{N}(u)\cup\{u\}$. \textit{The degree of the vertex $u$} is denoted by $\deg (u)$ and is the size of the neighbourhood set of $u$, that is $\deg u=\#(\mathcal{N}(u))$. A graph is called \textit{complete} if any two vertices are adjacent. We denote by $\mathcal{K}_n$ the complete graph with $n$ vertices. Moreover, we denote by $\mathcal{K}_{1,n}$ \textit{the star graph} on $n+1$ vertices, that is the graph with the vertex set $V=\{u,v_1,\ldots,v_n\}$ and the edges $\{u,v_i\}$, $1\leq i\leq n$. More generally,\textit{ the complete bipartite graph}, denoted by $\mathcal{K}_{m,n}$, has the vertex set $U\cup V$, with $U=\{u_1,\ldots,u_m\}$, $V=\{v_1,\ldots,v_n\}$, $U\cap V=\emptyset$ and the set of edges $\{\{u_i,v_j\}: 1\leq i\leq m, 1\leq j\leq n\}$.   
	
	By \textit{a subgraph} $H$ of $G$ we mean a graph with the property that $V(H)\subseteq V(G)$ and $E(H)\subseteq E(G)$. One says that a subgraph $H$ of $G$ is \textit{induced} if whenever $u,v\in V(H)$ so that $\{u,v\}\in E(G)$ then $\{u,v\}\in E(H)$. 
	
	\textit{A path of length }$t\geq2$ in $G$ is, by definition, a set of distinct vertices $u_0,u_1,\ldots,u_t$ such that $\{u_i,u_{i+1}\}$ are edges in $G$ for all $i\in\{0,\ldots,t-1\}$. \textit{The distance between two vertices $u$ and $v$}, denoted by $\d_G(u,v)$, is defined to be the length of a shortest path joining $u$ and $v$. If there is no path joining $u$ and $v$, then $\d_G(u,v)=\infty$. We will drop the subscript when the confusion is unlikely. \textit{The diameter of the graph $G$}, denoted by $\diam(G)$, is the maximum of all the distances between any two vertices in $G$, namely $\diam(G)=\max\{\d(u,v): \ u,v\in V(G)\}.$
	
	\textit{A cycle of length $n\geq3$}, usually denoted by $C_n$, is a graph with the vertex set $\{u_1,\ldots,u_n\}$ and the set of edges $\{u_i,u_{i+1}\}$, $1\leq i\leq i+1$, where $n+1=1$ by convention. A graph is \textit{chordal} if it does not have any induced cycles of length strictly greater than $3$.

	If $G$ is a finite simple graph, \textit{the line graph} of the graph $G$, denoted by $L(G)$, is defined to have as its vertices the edges of $G$, and two vertices in $L(G)$ are adjacent if the corresponding edges in $G$ share a vertex in $G$.

	\subsection{Hypergraphs}
	
	\textit{A hypergraph} $\Hc$ on the vertex set $V$ is a set of subsets of $V$ (called \textit{edges} of $\Hc$) such that if $e_1$ and $e_2$ are distinct edges of $\Hc$ then $e_1 \nsubseteq e_2$. For more details on hypergraphs and relations to simplicial complexes, one may check \cite{MV} for instance. A hypergraph is \textit{$d$-uniform} if every edge has exactly $d$ vertices. 
	
	Through this paper to any simplicial complex $\Delta$ we will associate a hypergraph with the same vertex set as $\Delta$ and with the edges given by the facets of $\Delta$. We denote by $\Hc(\Delta)$ this hypergraph. Since we will not be mainly interested in the structure of the simplicial complex, but more on the combinatorics of the associated hypergraph, we will simply say $\Delta$, but we will understand $\Hc(\Delta)$ whenever the confusion is unlikely. 
	
	\section{The line graph of a simplicial complex}
	In the literature, there are various generalizations of line graphs of graphs to line graphs of hypergraphs.  In \cite{BHS}, the authors defined the notion of $k$-line graph of a hypergraph $\Hc$ as being the graph with the vertex set given by the set of edges of the hypergraph, $\Ec(\Hc)$, and two vertices are adjacent if the intersection of the corresponding edges of $\Hc$ has at least $k$ elements. They denote the $k$-line graph of the hypergraph $\Hc$ by $L_k(\Hc)$.
	
	We will consider the above definition for the case of pure simplicial complexes, where the hypergraph has the vertex set given by the vertex set of the simplicial complex and the edges are the facets. More precisely, let $\Delta$ be a pure simplicial complex of dimension $d-1$, $d\geq2$, on the vertex set $V=\{x_1,\ldots,x_n\}$, with the facet set $\mathcal{F}(\Delta)=\{F_1,\ldots,F_r\}$, $r\geq1$. We will consider $\Hc(\Delta)$ as the hypergraph $\mathcal{H}$. The \textit{$(d-1)$-line graph of} $\Hc(\Delta)$ is the graph with the vertex set given by the facets of $\Delta$ and the set of edges $\{\{F_i,F_j\}:\ \#(F_i\cap F_j)=d-1\}$ (we must have equality due to the fact that simplicial complex is pure of dimension $d-1$). Since this is the only line graph that we will consider through this paper, we will simply refer to it as \textit{the line graph of the simplicial complex $\Delta$} and we will denote it by $\Lc(\Delta)$. 
	
	In order to avoid confusions, we will denote by ${v_1,\ldots,v_r}$ the vertices of $\Lc(\Delta)$, where the vertex $v_i$ corresponds to the facet $F_i$. Moreover, we will denote the edges of the hypergraph $\Hc(\Delta)$ by $\Ec(\Hc(\Delta))$, while the edges of the graph $G$ will be simply denoted by $E(G)$.

	\begin{Remark}\rm
		It is easily seen that the graph $\Lc(\Delta)$ does not depend (up to a relabeling of the vertices) on the labels of the facets of $\Delta$.
	\end{Remark}
	Note that both $\Delta$ and $\Delta^c$ have the same line graph, as the next result shows.
	\begin{Proposition}\label{Deltac}
		If $\Delta$ is a pure simplicial complex of dimension $d-1$ with $n>d+1$ vertices, then $\Lc(\Delta)$ and $\Lc(\Delta^c)$ coincide (up to the labeling of the vertices). 
	\end{Proposition}
	\begin{proof} It is clear that both graphs have the same number of vertices. Firstly we prove that each edge of $\Lc(\Delta)$ induces an edge in $\Lc(\Delta^c)$. Indeed, $\Delta$ is pure of dimension $d-1$ implies $\Delta^c$ is pure of dimension $n-d-1$. By using $F^c\cap G^c=(F\cup G)^c$ and $\#(F\cap G)=d-1$ we get $$\#(F^c\cap G^c)=\#((F\cup G)^c)=n-(d+1)=n-d-1.$$
		
		For the converse, one may note that $\left(F^c\right)^c=F$. Let $\{v_{F^c},v_{G^c}\}$ be an edge in $\Lc(\Delta^c)$. We have $$\#(F\cap G)=\#(\left(F^c\right)^c\cap\left(G^c\right)^c)=\#(\left(F^c\cup G^c\right)^c)=n-(n-d+1)=d-1.$$
	\end{proof}
	\begin{Remark}\rm One may note that the inequality $n>d+1$ from Proposition \ref{Deltac} is sharp. Indeed, it is clear that one should have $n\geq d$. If $n=d$ then $\Delta$ is a simplex, so it has only one facet. If $n=d+1$, we consider the following example: let $\Delta=\langle\{x_1,x_2,x_3\},\{x_2,x_3,x_4\}\rangle$. Then $\Delta^c=\langle\{x_4\},\{x_1\}\rangle$, so $\Lc(\Delta)$ has two vertices and one edge, while $\Lc(\Delta^c)$ has two isolated vertices.
	\end{Remark}
 It is easy to see that the line graph is not connected, in general, even if the simplicial complex $\Delta$ is connected. Therefore, we give a sufficient condition for the connectivity of the line graph of a simplicial complex.
\begin{Proposition}\label{connected}
	If $\Delta$ is a pure shellable simplicial complex, then $\Lc(\Delta)$ is connected.
\end{Proposition}
\begin{proof} Since $\Delta $ is shellable, there is an order of the facets $F_1,\ldots,F_r$ such that for all $1\leq j< i\leq r$ there is a vertex $x\in F_i\setminus F_j$ and some $k\in\{1,\ldots,i-1\}$ with $F_i\setminus F_k=\{x\}$. In particular $\#(F_i\cap F_k)=d-1$, therefore $\{v_{i},v_{k}\}$ is an edge in $\Lc(\Delta)$. Thus $\Lc(\Delta)$ is connected. 
\end{proof} 

In particular, we obtain an algebraic condition for the connectivity of the line graph:

\begin{Corollary}
	Let $\Delta$ be a pure simplicial complex of dimension $d-1$, on $n>d+1$ vertices and $S=\kk[x_1,\ldots,x_n]$. If $I(\Delta^c)\subseteq S$ has linear quotients, then $\Lc(\Delta)$ is connected.
\end{Corollary}
\begin{proof}
	The proof is straightforward. Indeed, $I_{\Delta^\vee}$ has linear quotients by Propositions~\ref{Aldualcomp}, therefore $\Delta$ is shellable according to  Theorem~\ref{shell}. The statement follows by Proposition ~\ref{connected}.
\end{proof}
One may note that the converse does not hold. There are simplicial complexes which are not even Cohen--Macaulay, but their line graph is connected, as the following example shows:

\begin{Example}\rm
	Let $\Delta$ be the simplicial complex on the vertex set $\{x_1,\ldots,x_7\}$ with the set of facets $\Fc(\Delta)=\{\{x_1,x_2,x_3\},\{x_2,x_3,x_4\},\{x_3,x_4,x_5\},\{x_4,x_5,x_6\},\{x_5,x_6,x_7\}\}$. Therefore $\Delta$ and its line graph are
	\[\]
	\begin{center}
		\begin{figure}[h]
			\includegraphics[height=2.5cm]{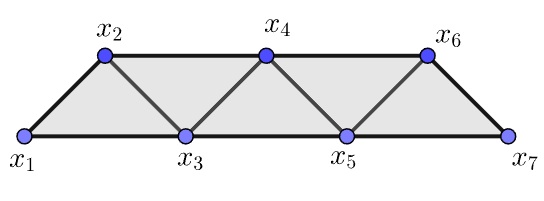}
			\caption{The simplicial complex $\Delta$}
				\end{figure}
		\end{center}
		\begin{center}
			\begin{figure}[h]
			\includegraphics[height=1.3cm]{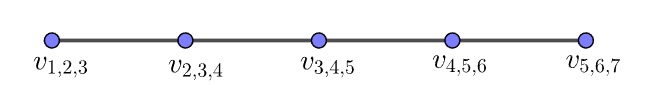}\caption{The line graph of $\Delta$}
		\end{figure}
	\end{center}
	Note that $\Lc(\Delta)$ is connected, but $\Delta$ is not even Cohen--Macaulay since the Stanley--Reisner ideal of its Alexander dual does not have a linear resolution as one may easy check with Singular, for instance \cite{DGPS}.
\end{Example}

The next problem naturally arise:
\begin{Problem}\rm
	Is there any characterization of those simplicial complexes whose line graph is connected?
\end{Problem}
	We consider the number of edges of the line graph. We recall that, for line graphs of graphs, the number of edges is known:
	\begin{Proposition}\cite[Proposition 7.6.2]{Vi}
		If $G$ is a graph with vertices $v_1,\ldots,v_n$ and edge set $E(G)$, then the number of edges of the line graph $L(G)$ is given by 
		\[\#(E(L(G)))=\sum\limits_{i=1}^n{\deg(v_i)\choose 2}=-\#(E(G))+\sum\limits_{i=1}^n\frac{\deg^2 v_i}{2}\]
	\end{Proposition}
	
	We will determine the number of edges of the line graph of a pure simplicial complex. Let $\Delta$ be a pure simplicial complex of dimension $d-1$ with the facet set $\Fc(\Delta)=\{F_1,\ldots,F_r\}$ 
	For each $i$, let $$s_i=\#(\{F_j:j>i, \#(F_j\cap F_i)=d-1\}).$$
	
	\begin{Remark}\rm
		Note that $s_i$ is just the number of neighbors of $v_i$ which were not counted before.
	\end{Remark}
	\begin{Proposition}Under the above assumptions,
		$\#(E(\Lc(\Delta)))=\sum\limits_{i=1}^r s_i$
	\end{Proposition}
	\begin{proof} Let $F_1,\ldots,F_r$ be a labeling of the facets of $\Delta$. An edge of the graph $\Lc(\Delta)$ is given by a pair of facets $F_i, F_j$ such that $\#(F_i\cap F_j)=d-1$. Therefore, the number of edges induced by the facet $F_i$ is given by all its neighbors except the ones which were considered before (in order to skip the overlaps). 
	\end{proof}
	
	\begin{Corollary}
		$\sum\limits_{i=1}^r \deg(v_i)=2\sum\limits_{i=1}^r s_i$
	\end{Corollary}
	\begin{proof}The equality follows from the Euler's inequality $2\#(E(G))=\sum\limits_{i=1}^n\deg(v_i)$ and the previous result.
	\end{proof}
The next result is just a reformulation  of \cite[Lemma 3.1]{BHS} and \cite[Remark 3.2]{BHS} in terms of our definition of line graphs of simplicial complexes (in order to not complicate the terminology, we will not recall their statements here).

\begin{Lemma}\label{d=2}\cite[Lemma 3.1]{BHS} Let $\Delta$ be a pure simplicial complex of dimension $d-1$ and $G=\Lc(\Delta)$. If $v_1,v_2$ are two vertices of $G$ such that $\d(v_1,v_2)=2$, then $\#(F_1\cap F_2)=d-2$.
\end{Lemma}
\begin{proof}
	The results holds for $p=1$ and $h=d$ in \cite[Lemma 3.1]{BHS} and \cite[Remark 3.2]{BHS}.
\end{proof}
We will often use the behaviour of the line graph on the deletion of vertices. Therefore, we describe it in the sequel.
\begin{Lemma}\label{deletion}
	Let $\Delta$ be a pure simplicial complex of dimension $d-1$ and $F$ a facet of $\Delta$. Let $\Delta'$ be the subcomplex of $\Delta$ with $\Fc(\Delta')=\Fc(\Delta)\setminus\{F\}$. Then $\Lc(\Delta')=\Lc(\Delta)\setminus\{v_F\}$, where $v_F$ is the corresponding vertex of the facet $F$.
\end{Lemma}
\begin{proof}
	The proof is straightforward since $\Lc(\Delta)\setminus\{v_F\}$ is the subgraph of $\Lc(\Delta)$ obtained by deleting the vertex $v_F$ and all the edges which are incident to $v_F$. Thus it is clear that both graphs $\Lc(\Delta')$ and $\Lc(\Delta)\setminus \{v_F\}$ have the same set of vertices. Moreover, it is easy to see that $E(\Lc(\Delta'))\subseteq E(\Lc(\Delta)\setminus\{v_F\})$. For the other inclusion, one has to note that in $\Lc(\Delta)\setminus\{v_F\}$ any edge is of the form $\{v_{F_i},v_{F_j}\}$ such that $\#(F_i\cap F_j)=d-1$ and $F_i\neq F$, $F_j\neq F$. Hence, they are also edges in $\Lc(\Delta')$.
\end{proof}
Note that the converse does not hold. There are graphs which are line graphs of some simplicial complex, but the graph obtained by adding a vertex is not a line graph, as the following result shows:
\begin{Proposition}\cite[Theorem 4.3]{BHS}\label{wheel} Let $W_n$ the graph which is a wheel with a central vertex $v$, joined to every other vertex $u_i$, $1\leq i\leq n-1$, of a cycle of length $n-1$. Thus, for any $k$, $k\geq 3$,
	\begin{itemize}
		\item[(i)] $W_{2k}$ is not the line graph of any pure simplicial complex.
		\item[(ii)] Any proper induced subgraph of $W_{2k}$ is the line graph of some pure simplicial complex. 
	\end{itemize}
\end{Proposition} 
It is also useful to consider the behaviour of line graphs with respect to induced subgraphs.
\begin{Proposition}\label{inducedsubgraph}
	Let $G$ be a finite simple graph and $G'$ an induced subgraph of $G$. If there is no pure simplicial complex $\Delta'$ such that $\Lc(\Delta')=G'$, then $G$ is not the line graph of any pure simplicial complex. 
\end{Proposition}
\begin{proof} Assume by contradiction that there is a pure simplicial complex such that $\Lc(\Delta)=G$. Let $V(G)\setminus V(G')=\{v_1,\ldots,v_r\}$ and $F_1,\ldots,F_r$ the corresponding facets. We consider the simplicial complexes $\Delta_1,\ldots, \Delta_r$ such that $\Fc(\Delta_1)=\Fc(\Delta)\setminus\{F_1\}$, and $\Fc(\Delta_i)=\Fc(\Delta_{i-1})\setminus\{F_i\}$ for all $2\leq i\leq r$. Then, by applying repeatedly Lemma \ref{deletion}, one gets \[\Lc(\Delta_r)=\Lc(\Delta_{r-1})\setminus\{v_r\}=\cdots=\Lc(\Delta)\setminus\{v_1,\ldots,v_r\}=G\setminus\{v_1,\ldots,v_r\}=G'\]a contradiction since $G'$ is not the line graph of any pure simplicial complex.
	\end{proof}
\begin{Proposition}\label{degclique}
	Let $G$ be a finite simple connected graph and $\Delta$ a pure simplicial complex of dimension $d-1$ such that $\Lc(\Delta)=G$. Let $v$ be a vertex of $G$. Then $v$ belongs to at most $\min\{\deg(v),d\}$ maximal cliques.
\end{Proposition}
\begin{proof}
	Let us assume that $\deg(v)=r$, $\mathcal{N}(v)=\{v_1,\ldots,v_r\}$, and $F_1,\ldots,F_r$ the corresponding facets. Let $F=\{x_1,\ldots,x_d\}$ be the corresponding facet of $v$. Then each $F_i$ is obtained from $F$ by removing a vertex and adding a new one. It is clear that one has at most $d$ vertices to remove. If $r\leq d$, the statement is clear. If $r>d$, we will prove that the minimum is $d$. Indeed, we assume for instance that there are two facets $F_i$ and $F_j$ such that $F\setminus F_i=F\setminus F_j=\{x_i\}$. Then $F_i=(F\setminus\{x_i\})\cup\{y_i\}$ and $F_j=(F\setminus\{x_i\})\cup\{y_j\}$ which implies that $\#(F_i\cap F_j)=d-1$ and $\{v_i,v_j\}\in E(G)$. Thus $\{v,v_i,v_j\}$ form a complete graph. The statement follows.
\end{proof}

Note that the converse implication does not hold. More precisely, there are graphs with the property that each vertex belongs to at most $D=\min\{\deg(v),d\}$ for any $d\geq2$, but there is no simplicial complex $\Delta$ of dimension $d-1$ such that $\Lc(\Delta)=G$ as the following example shows.

\begin{Example}\label{ex1k23}\rm
	Let $G=\mathcal{K}_{2,3}$ with the edges $\{u_i,v_j\}$, $1\leq i\leq 2$, $1\leq j\leq 3$. 
	
		\begin{center}
		\begin{figure}[h]
			\includegraphics[height=3.5cm]{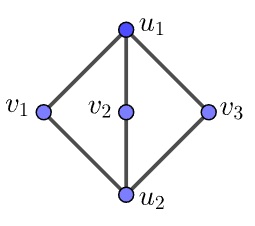}
			\caption{The complete bipartite graph $\mathcal{K}_{2,3}$}
		\end{figure}
	\end{center} 
	Each edge form a maximal clique. Thus $u_i$ is contained in $3$ maximal cliques, $1\leq i\leq 2$ and $v_j$ is contained in two maximal cliques, $1\leq j\leq 3$, but there is not any pure simplicial complex $\Delta$ such that $\Lc(\Delta)=G$, by \cite[Remark 4.3]{BHS}. 
\end{Example}

We end up this section with a result which extends the results obtained in the next sections to higher dimensions.

\begin{Proposition}
	Let $G$ be a finite simple graph, $\Delta=\langle F_1,\ldots,F_r\rangle$ a pure $(d-1)$-dimensional simplicial complex such that $\Lc(\Delta)=G$. Let $\Gamma=\Delta*\langle\{y_1,\ldots,y_{N}\}\rangle$, where $V(\Delta)\cap\{y_1,\ldots,y_N\}=\emptyset$. Then $\Lc(\Gamma)=G$.
\end{Proposition}
\begin{proof}
	One may note that $\Gamma$ is a pure simplicial complex of dimension $d+N-1$ whose facets are $G_i=F_i\cup \{y_1,\ldots,y_N\}$, for all $1\leq i\leq r$. It is obvious that $\#(G_i\cap G_j)=d+N-1$ if and only if $\#(F_i\cap F_j)=d-1$ for all $1\leq i<j\leq r$. Therefore $\{G_i,G_j\}$ gives an edge in $\Lc(\Gamma)$ if and only if $\{F_i,F_j\}$ gives an edge in $\Lc(\Delta)$, for all $1\leq i<j\leq r$. The statement follows.
\end{proof}
\section{Cycles and complete graphs as line graphs of simplicial complexes}
In the sequel, we completely describe those pure simplicial complexes whose line graphs are  complete graphs. Moreover, we give classes pure simplicial complexes whose line graph is a cycle.
\begin{Proposition}\label{complete}
	Let $\Delta$ be a pure simplicial complex of dimension $d-1$ and $\Lc(\Delta)$ its line graph. Then $\Lc(\Delta)$ is a complete graph $\mathcal{K}_n$, for some $n\geq 3$, if and only if one of the following holds
	\begin{itemize}
		\item[(i)] $\Delta$ is of the form $\langle\{x_1\},\ldots,\{x_n\}\rangle*\Gamma$, where $\Gamma$ is a $(d-2)$-simplex such that $\{x_1,\ldots,x_{n}\}\cap V(\Gamma)=\emptyset$ 
		\item[(ii)] the facets of $\Delta$ are $(d-1)$-faces of a $d$-simplex. $(n\leq d+1)$.
	\end{itemize} 
\end{Proposition}
\begin{proof}``$\Leftarrow$"
	If $\Delta$ is one of the above simplicial complexes, then $\#(F\cap F')=d-1$ for any $F,F'\in\Fc(\Delta)$, hence $\Lc(\Delta)$ is a complete graph.
	
	``$\Rightarrow$" Assume that $\Lc(\Delta)$ is a complete graph $\mathcal{K}_n$, $n\geq3$. We have to prove that $\Delta$ has one of the forms given in (i) and (ii). We use induction on the number of vertices of $\Lc(\Delta)$ (or on the number of facets of $\Delta$, equivalently).
	
	For $n=3$, $\Lc(\Delta)$ is $C_3$, that is $\Delta=\langle F_1,F_2,F_3\rangle$ with $$\#(F_1\cap F_2)=\#(F_1\cap F_3)=\#(F_2\cap F_3)=d-1.$$ By the inclusion-exclusion principle we have $\#(F_1\cup F_2\cup F_3)-\#(F_1\cap F_2\cap F_3)=3$. Hence the facets of $\Delta$ should be in one of the cases:
	\begin{itemize}
		\item $F_1=\{x_1\}\cup V(\Gamma)$, $F_2=\{x_2\}\cup V(\Gamma)$,$F_3=\{x_3\}\cup V(\Gamma)$, where $\Gamma$ is a $(d-2)$-simplex such that $\{x_1,x_2,x_3\}\cap V(\Gamma)=\emptyset$.
		\item  $F_1= \{x_1,x_2\}\cup V(\Gamma)$, $F_2=\{x_1,x_3\}\cup V(\Gamma)$, $F_3=\{x_2,x_3\}\cup V(\Gamma)$, where $\Gamma$ is a $(d-3)$-simplex such that $x_1,x_2,x_3\notin V(\Gamma)$. In particular, the facets of $\Delta$ are $(d-1)$-faces of a $d$-simplex $\Gamma_1$ with $V(\Gamma_1)=V(\Gamma)\cup\{x_1,x_2,x_3\}$.
	\end{itemize} 
	
	Assume that the statement is true for $n\geq 3$ vertices and we prove it for $n+1$ vertices. Let  $\Fc(\Delta)=\{F_1,\ldots,F_{n+1}\}$ with $\Lc(\Delta)$ a complete graph, and we consider the simplicial complex $\Delta'=\Delta\setminus \{F_{n+1}\}$. By Lemma~\ref{deletion} we have $\Lc(\Delta')=\Lc(\Delta)\setminus\{v_{n+1}\}$, hence $\Lc(\Delta')$ is a complete graph with $n$ vertices. By the induction hypotheses, we have one of the following cases:
	
	\textbf{Case 1}: $\Delta'=\langle\{x_1\},\ldots,\{x_n\}\rangle*\Gamma$, where $\Gamma$ is a $(d-2)$-simplex such that $\{x_1,\ldots,x_{n}\}\cap V(\Gamma)=\emptyset$. We may assume that $F_i=\{x_i\}\cup V(\Gamma)$, for all $1\leq i\leq~ n$. Since $\#(F_{n+1}\cap F_1)=\#(F_{n+1}\cap F_2)=d-1$, and $x_1,x_2\notin V(\Gamma)$, we have that $V(\Gamma)\subset F_{n+1}$. Since $\#(V(\Gamma))=d-1$, we must have $F_{n+1}=\{x_{n+1}\}\cup V(\Gamma)$ for some $x_{n+1}\notin V(\Gamma)$. Therefore $\Delta=\langle\{x_1\},\ldots,\{x_{n+1}\}\rangle*\Gamma$, where $\Gamma$ is a $(d-2)$-simplex such that $\{x_1,\ldots,x_{n+1}\}\cap V(\Gamma)=\emptyset$.
	
	\textbf{Case 2}: The facets of $\Delta'$ are $(d-1)$-faces of a $d$-simplex. If we assume by contradiction that $n=d+1$ then one must have $\{x_1,\ldots,x_{d+1}\}\subset F_{n+1}$ since $\#(F_{n+1}\cap F_i)=d-1$ for all $i\in\{1,\ldots,n\}$, a contradiction. Hence we must have $n\leq d$. In this case we may assume, up to relabelling of the vertices of $\Gamma$, that $F_i=\{x_1,\ldots,\widehat{x}_i,\ldots, x_{d+1}\}$. Since $\Lc(\Delta)$ is a complete graph, we must have $\#(F_{n+1}\cap F_i)=d-1$ for all $1\leq i\leq n$. Assume by contradiction that $F_{n+1}$ is not a $(d-1)$-face of $\Gamma$. Hence there is a vertex $x\in F_{n+1}\setminus V(\Gamma)$. In particular $x\in F_{n+1}\setminus F_i$, for all $i$. One may easy note that this implies $F_{n+1}=\left(F_1\cup\left\{x\right\}\right)\setminus\{x_i\}=\left(F_2\cup\left\{x\right\}\right)\setminus\{x_j\}=\left(F_3\cup\left\{x\right\}\right)\setminus\{x_k\}$ for some $i,j,k\in\{1,\ldots,d+1\}$, $i\neq 1$, $j\neq 2$, $k\neq 3$. In particular $F_1\setminus\{x_i\}=F_2\setminus\{x_j\}=F_3\setminus\{x_k\}$. Taking into account that $F_1=\{x_2,\ldots,x_{d+1}\}$, $F_2=\{x_1,x_3,\ldots,x_{d+1}\}$, and $F_3=\{x_1, x_2,x_4,\ldots,x_{d+1}\}$, the equality $F_1\setminus\{x_i\}=F_2\setminus\{x_j\}$ gives $i=2$ and $j=1$, while the equality $F_1\setminus\{x_i\}=F_3\setminus\{x_k\}$ gives $i=3$ and $k=1$, a contradiction. Therefore $F_{n+1}$ is a $(d-1)$-face of a $d$-simplex.
\end{proof}

Next we will find classes of simplicial complexes whose line graph is a cycle. The description of cycles of length 3 can be easily obtain from the induction step in the proof of Proposition~\ref{complete} since $C_3$ can be also considered as a complete graph with $3$ vertices. Since its particular description will be used later, we formulate it in the next result.
\begin{Corollary}\label{C3}
	Let $\Delta$ be a pure simplicial complex of dimension $d-1$ with three facets. Then $\Lc(\Delta)$ is $C_3$ if and only if one of the following holds
	\begin{itemize}
		\item[a)]  $\Delta$ is of the form $\langle\{x_1\},\{x_2\},\{x_3\}\rangle*  \Gamma$, where $\Gamma$ is a $(d-2)$-simplex such that $x_1,x_2,x_3\notin V(\Gamma)$
		\item[b)] $\Delta$ is of the form $\langle \{x_1,x_2\},\{x_1,x_3\},\{x_2,x_3\}\rangle* \Gamma$, where $\Gamma$ is a $(d-3)$-simplex such that $x_1,x_2,x_3\notin V(\Gamma)$.
\end{itemize} \end{Corollary}

We consider now $C_r$, for $r\geq4$, as line graphs of some simplicial complex.
\begin{Proposition}\label{cycle}
	Let $\Delta$ be a pure simplicial complex of dimension $d-1$ and $\Lc(\Delta)$ its line graph. If one of the following holds:
	\begin{itemize}
		\item if $d<r-1$, the facets of $\Delta$ are the $(d-1)$-paths of the cycle of length $r$.
		\item if $d\geq r-1$ the facets of $\Delta$ are the $(r-3)$-paths of the cycle of length $r$ union a set $H$ of cardinality $d-r+2$. 
	\end{itemize}then $\Lc(\Delta)$ is a cycle of length $r\geq4$. 
\end{Proposition}

\begin{proof} It is clear by the shape of the facets that every vertex in $\Lc(\Delta)$ has exactly two neighbours.
\end{proof}

Note that this is not a complete characterizations. There are $(d-1)$-simplicial complexes whose line graph is a cycle but the set of vertices does not verify any of the conditions from Proposition \ref{cycle}, as the following example shows.

\begin{Example}\rm Let $\Delta$ be the simplicial complex with the set of facets $$\Fc(\Delta)=\{\{x_1,x_2,x_3\},\{x_2,x_3,x_4\},\{x_3,x_4,x_5\}, \{x_1,x_4,x_5\}, \{x_1,x_4,x_6\}, \{x_1,x_2,x_6\}\}.$$ It is easy to see that $\Lc(\Delta)$ is a cycle of length $6$, but the facets of $\Delta$ are not 2-paths of the cycle $C_6$ with the set of vertices $\{x_1,\ldots,x_6\}$ (arranged in any order).
\end{Example}

\section{Complete bipartite graphs and line graphs of simplicial complexes}

According to \cite[Remark 4.3]{BHS}, it is clear that not all the complete bipartite graphs are line graphs of some simplicial complex. We will characterize all the complete bipartite graphs which have this property. Moreover, we describe all the pure simplicial complexes whose line graphs are complete bipartite.

We start by considering star graphs which can be viewed as complete bipartite graphs of the form $\mathcal{K}_{1,n}$, $n\geq 2$.
\begin{Proposition}\label{1,d+1}
	If $G$ is the line graph of a pure simplicial complex of dimension $d-1$, then $G$ does not contain $\mathcal{K}_{1,d+1}$ as an induced subgraph. 
\end{Proposition}
\begin{proof} Assume by contradiction that $\mathcal{K}_{1,d+1}$ is an induced subgraph of $G$. Then there is a vertex $v_i$ with $d+1$ neighbours. Since $F_i$ has $d$ elements, there are two facets $F_{j}$ and $F_{k}$ such that $F_j\cap F_i=F_k \cap F_i$ and $v_j$, $v_k$ are neighbors of $v_i$. Therefore $\#(F_j\cap F_k)=d-1$, hence $\{v_j,v_k\}$ is an edge in $\Lc(\Delta)$, a contradiction.
\end{proof}
We can describe the shape of pure simplicial complexes whose line graphs are star graphs:
\begin{Proposition}\label{k1n} Let $\Delta$ be a pure simplicial complex of dimension $d-1$ and $n\leq d$ an integer. The following are equivalent:
	\begin{itemize}
		\item[(i)] $\mathcal{K}_{1,n}$ is the line graph of $\Delta$.
		\item[(ii)] $\Fc(\Delta)=\{\{x_1,\ldots,x_d\}\}\cup\{\{x_1,\ldots,\widehat{x}_i,\ldots,x_d,x_{d+i}\}:1\leq i\leq n\}$. 
	\end{itemize} 
\end{Proposition}
\begin{proof}
	``(i)$\Rightarrow$(ii)" In order to construct a pure simplicial complex of dimension $d-1$ whose line graph is $\mathcal{K}_{1,n}$ we firstly note that $n\leq d$, by Proposition \ref{1,d+1}. Secondly, if we consider that $\{v_1,\ldots,v_{n+1}\}$ is the set of vertices of $\mathcal{K}_{1,n}$ and the edges are $\{v_1,v_i\}$, for all $i\in\{2,\ldots,n+1\}$, then $\#(F_1\cap F_i)=d-1$ for all $2\leq i\leq n+1$ and $\#(F_i\cap F_j)=d-2$ for all $2\leq i<j\leq n+1$ by Lemma \ref{d=2}, where $\Fc(\Delta)=\{F_1,\ldots,F_{n+1}\}$. Therefore, if $F_1=\{x_1,\ldots,x_d\}$, then, for instance, $F_2=\{x_2,\ldots,x_d,x_{d+1}\}$. In order to fulfill the conditions, $F_3=\{x_1,x_3,\ldots,x_{d},x_{d+2}\}$. We continue like this until we construct all the facets. Therefore, $\Delta$ has the vertex set $\{x_1,\ldots,x_{d+n}\}$ and the facets set $$\Fc(\Delta)=\{\{x_1,\ldots,x_d\}\}\cup\{\{x_1,\ldots,\widehat{x}_i,\ldots,x_d,x_{d+i}\}:1\leq i\leq n\}.$$
	
	``(ii)$\Rightarrow$(i)" It is easy to note that the $\Lc(\Delta)=\mathcal{K}_{1,n}$
\end{proof} 
 We consider now complete bipartite graphs with $m=2$:
\begin{Proposition}\label{complete-bipartite-2n}
	The complete bipartite graph $\mathcal{K}_{2,n}$ is not the line graph of any pure simplicial complex, for any $n\geq 3$.
\end{Proposition}
\begin{proof}
We use induction for proving the statement.
	
	If $n=3$ it is easy to check that $\mathcal{K}_{2,3}\neq\Lc(\Delta)$ for any pure simplicial complex $\Delta$ (see also \cite[Remark 4.3]{BHS}). If $n>3$, then $\mathcal{K}_{2,3}$ is an induced graph of $\mathcal{K}_{2,n}$. Therefore the statement follows by Proposition \ref{inducedsubgraph}.  
\end{proof} 
\begin{Proposition}\label{k22}
	Let $m,n>1$ be integers, $\mathcal{K}_{m,n}$ the complete bipartite graph and $\Delta$ a pure simplicial complex of dimension $d-1$. Then $\mathcal{K}_{m,n}=\Lc(\Delta)$ if and only if $m=n=2$.
In this case
$\Delta$ has the form $$\langle\{x_1,x_2\},\{x_2,x_3\},\{x_3,x_4\},\{x_4,x_1\}\rangle*\langle\{x_5,\ldots,x_{d+2}\}\rangle. $$
\end{Proposition}
\begin{proof}
	``$\Rightarrow$" It is easy to see that it is enough to study only the values of $m$ since a similar discussion holds for $n$. If $m=2$ it follows by Proposition \ref{complete-bipartite-2n} that $n=2$. For $m\geq 3$, we assume by contradiction that $\Lc(\Delta)=\mathcal{K}_{m,n}$. Since $\mathcal{K}_{3,2}$ is an induced subgraph of $\mathcal{K}_{m,n}$, we get a contradiction.
	
	``$\Leftarrow$" We assume that $m=n=2$, that is $\mathcal{K}_{2,2}=C_4$. We construct a pure simplicial complex of dimension $d-1$ whose line graph is $\mathcal{K}_{2,2}$. We consider that $\mathcal{K}_{2,2}=C_4$ has the vertex set $\{v_1,v_2,v_3,v_4\}$ with the edges $\{v_1,v_2\},\{v_2,v_3\},\,\{v_3,v_4\},\{v_4,v_1\}.$ Let $\Delta$ be the simplicial complex with the facet set $\Fc(\Delta)=\{F_1,F_2,F_3,F_4\}$. We assume that the facet $F_i$ correspond to the vertex $v_i$, $1\leq i\leq 4$ and that the facet $F_1$ is $F_1=\{x_1,\ldots,x_d\}$. Since $\d(v_1,v_2)=1$ and $\d(v_1,v_3)=2$ we must have $\#(F_1\cap F_2)=d-1$ and $\#(F_1\cap F_3)=d-2$ by Lemma \ref{d=2}. Thus $F_2=\{x_1,\ldots,\widehat{x}_i,\ldots,x_d,x_{d+1}\}$ and $F_3=\{x_1,\ldots,\widehat{x}_j,\ldots,\widehat{x}_k,\ldots,x_{d+2}\}$ for some  $1\leq i\leq d$ and $1\leq j\leq k\leq d+1$. Since $\#(F_1\cap F_3)=d-2$, we must have $x_{d+1}\neq x_{d+2}$. 
	
	We assume that $k=d+1$. Since $F_3\setminus F_2=\{x_i,x_{d+2}\}$ and $\#(F_2\cap F_3)=d-1$, we must have $i=j$, that is $F_3=\{x_1,\ldots,\widehat{x}_i,\ldots,x_d,x_{d+2}\}$. But this implies $\#(F_1\cap F_3)=d-1$, a contradiction. Therefore $k\leq d$.
	
	If $i\neq j,k$ we have $F_2\setminus F_3=\{x_j,x_{k}\}$ and $\#(F_2\cap F_3)=d-1$, a contradiction. We must have $i=j$ or
$i=k$. Without losing the generality, we assume $i=j$, so $F_3=\{x_1,\ldots,\widehat{x}_i,\ldots,\widehat{x}_k,\ldots,x_{d+1},x_{d+2}\}$.

We have to construct $F_4$. Since $\d(v_1,v_4)=1$, we have $F_4=\{x_1,\ldots,\widehat{x}_l,\ldots,x_d,y\}$, for some vertex $y$ and some $1\leq l\leq d$. We will prove that $l=k$ and $y=x_{d+2}$.

Firstly we assume that $y=x_{d+1}$, that is $F_4=\{x_1,\ldots,\widehat{x}_l,\ldots,x_d,x_{d+1}\}$. Since $F_4\cap F_2=\{x_1,\ldots,x_d,x_{d+1}\}\setminus\{x_i,x_l\}$ and $\d(v_2,v_4)=2$, we have $\#(F_4\cap F_2)=d-1$,that is $\d(v_2,v_4)=1$, a contradiction. Thus, $y\neq x_{d+1}$.

Now, $F_4\cap F_2=\{x_1,\ldots,x_d\}\setminus\{x_i,x_l\}$ and $\#(F_2\cap F_4)=d-2$ imply $i\neq l$. But $F_4\setminus F_3=\{x_i,x_k,y\}$ and $\#(F_3\cap F_4)=d-1$ imply $k=l$ and $y=x_{d+2}$, that is  $F_4=\{x_1,\ldots,\widehat{x}_k,\ldots,x_d,x_{d+2}\}$.

Therefore $\Delta$ has the vertex set $\{x_1,\ldots,x_{d+2}\}$ and $$\Delta=\langle\{x_i,x_k\},\{x_k,x_{d+1}\},\{x_{d+1},x_{d+2}\},\{x_i,x_{d+2}\}\rangle*\langle\{x_1,\ldots,\widehat{x}_i,\ldots,\widehat{x}_k,\ldots,x_{d}\}\}\rangle.$$
\end{proof}

We obtained a characterization of all pure simplicial complex of the simplicial complexes whose line graph is $C_4$.

\begin{Corollary}\label{C4}
	Let $\Delta$ be a pure simplicial complex of dimension $d-1$. Then $\Lc(\Delta)=C_4$ if and only if
 $\Delta$ has the form $$\langle\{x_1,x_2\},\{x_2,x_3\},\{x_3,x_4\},\{x_4,x_1\}\rangle*\langle\{x_5,\ldots,x_{d+2}\}\}\rangle .$$
\end{Corollary}

\begin{proof}
	``$\Rightarrow$" Follows by Proposition \ref{k22}.
	
	``$\Leftarrow$" By the shape of the facets it is easy to see that $\Lc(\Delta)=C_4$.
\end{proof}

By using the above result, we can characterize all complete bipartite graphs which are line graphs of some pure simplicial complex.
\begin{Theorem}\label{complete-bipartite-mn}
	Let $m,n\geq1$ be integers, $\mathcal{K}_{m,n}$ the complete bipartite graph and $\Delta$ a pure simplicial complex of dimension $d-1$. Then $\mathcal{K}_{m,n}=\Lc(\Delta)$ if and only if one of the following holds:
	\begin{itemize}
		\item[(i)] $m=1$ and $n\leq d$ $($or $n=1$ and $m\leq d)$. In this case $\Delta$ has the vertex set $\{x_1,\ldots,x_{d+n}\}$ and the set of facets  $$\Fc(\Delta)=\{\{x_1,\ldots,x_d\}\}\cup\{\{x_1,\ldots,\widehat{x}_i,\ldots,x_d,x_{d+i}\}:1\leq i\leq n\}.$$ \item[(ii)] $m=n=2.$
		In this case $\Delta$ has the form$$\langle\{x_1,x_2\},\{x_2,x_3\},\{x_3,x_4\},\{x_4,x_1\}\rangle*\langle\{x_5,\ldots,x_{d+2}\}\}\rangle. $$
	\end{itemize} 
\end{Theorem}
\begin{proof}
	The result follows by Propositions \ref{k22} and \ref{k1n} 
\end{proof}
 \section{A particular class of graphs}
 In graph theory, there is also a different generalization of line graph of a graph for hypergraphs. This generalization, also called the intersection graph, is defined as follows: given a hypergraph $\Hc$, the line graph of $\Hc$ has the vertex set given by the edges, and two vertices are adjacent if the intersection of the corresponding edges is nonempty. In order to not create confusion, we will refer to it as the \textit{intersection graph}.
 A particular class of hypergraphs which is of interest in this context is that of linear $d$-uniform hypergraphs which are $d$-uniform hypergraphs with the intersection of any two edges of cardinality at most one. In \cite{NRSS} there is constructed a class of forbidden induced subgraphs for intersection graphs of linear $d$-uniform hypergraphs.
 
  Let $G_1(t)$ be the class of graphs obtained by arranging $t+2$ copies of $\mathcal{K}_4-e$ (the complete graph of order $4$ less an edge
 $e$), in the form of a chain to get a graph with maximum degree less than or equal to $4$ and attaching two pendant edges at each of the two degree two vertices of the graph thus obtained, \cite[pp 159]{NRSS}.
 
 \begin{figure}[h]
 	\includegraphics[height=2cm]{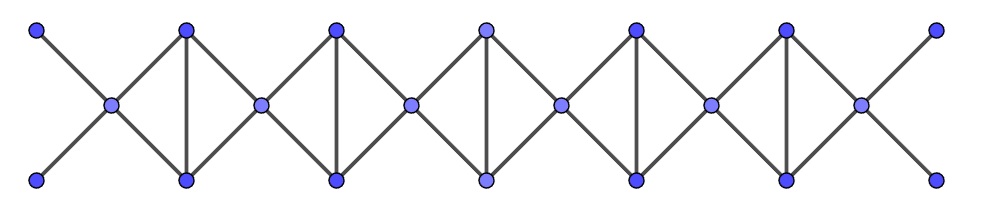}
 	\caption{The graph $G_1(3)$}
 \end{figure}
 
\newpage We will prove that this is not a forbidden class for the line graphs of simplicial complexes.
 
 \begin{Theorem}\label{G1t} Let $t\geq1$ be an integer and $G_1(t)$ the class of graphs described above. Then there is a $2$-simplicial complex $\Delta$ such that $\Lc(\Delta)=G_1(t)$.\end{Theorem}
 	\begin{proof}
 	We consider the graph $G_1(t)$ with the following labels of the vertices:
 	
 	\begin{figure}[h]
 		\includegraphics[height=4.5cm]{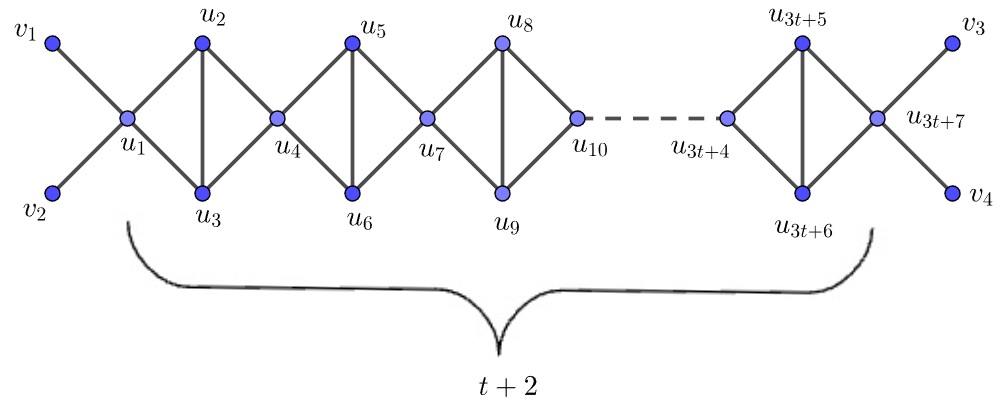}
 		\caption{The graph $G_1(t)$}
 	\end{figure}
 	
 	We start defining the facets of $\Delta$ by using Corollary \ref{C3}. Let us construct the facets of the cycle $\{u_1,u_2,u_3\}$ by using subcase (a) of Corollary \ref{C3}, namely $F_1=\{x_1,x_4,x_5\}$, $F_2=\{x_2,x_4,x_5\}$, $F_3=\{x_3,x_4,x_5\}$. Then the corresponding facets of $v_1$ and $v_2$ are $G_1=\{x_1,x_4,y_1\}$, $G_2=\{x_1,x_5,y_2\}$. Since $\{u_2,u_4\}$ and $\{u_3,u_4\}$ are edges, we must have $F_4=\{x_4,x_5,x_6\}$. Now we continue by using the second case of Corollary \ref{C3}: $F_5=\{x_4,x_6,x_7\}$, and $F_6=\{x_5,x_6,x_7\}$. Hence $F_7=\{x_6,x_7,x_8\}$. We continue like this: $F_8=\{x_6,x_8,x_9\}$, and $F_9=\{x_7,x_8,x_9\}$. Hence $F_{10}=\{x_8,x_9,x_{10}\}$. It is clear that we can use this method until we construct the corresponding facet of $u_{3t+7}$, let's say $F_{3t+7}=\{x_{\alpha_1},x_{\alpha_2},x_{\alpha_3}\}$. If we look at the construction of $F_4$, $F_7$, $F_{10}$, we note that $x_{\alpha_3}$ is a new vertex, so we can consider the corresponding facets of $v_3$ and $v_4$ as $G_3=\{x_{\alpha_1},x_{\alpha_3},y_3\}$, and $G_4=\{x_{\alpha_2},x_{\alpha_3},y_4\}$.
  We proved that, for the simplicial complex $\Delta$ with $\Fc(\Delta)=\{F_1,\ldots,F_{3t+7},G_1,G_2,G_3,G_4\}$ we have $\Lc(\Delta)=G_1(t)$.
 	\end{proof}

	\section{Graphs and line graphs of a simplicial complex}
Now we pay attention to those graphs which are line graphs of a pure simplicial complex of a fixed dimension. In this case, we have to recall that not every graph is the line graph of a graph. In our context, a graph is a simplicial complex of dimension 1 with the facets given by the edges of the graph. Therefore, knowing the dimension of the simplicial complex, one has to determine the properties that a graph should have in order to be the line graph of a simplicial complex. For line graphs of graphs there is the following characterization:
\begin{Theorem}\cite{B}\label{LG} The following statements are equivalent for a graph $G$.
	\begin{itemize}
		\item[(i)] $G$ is the line graph of some graph
		\item[(ii)] The edges of $G$ can be partitioned into complete subgraphs in such a way that no vertex belongs to more than two of the subgraphs.
		\item[(iii)] The graph $\mathcal{K}_{1,3}$ is not an induced subgraph of $G$; and if $abc$ and $bcd$ are distinct odd triangles, then $a$ and $d$ are adjacent (we recall that a triangle is \textit{odd} if there is a vertex of $G$ which is adjacent to an odd number of vertices of the triangle).
		\item[(iv)] None of the nine graphs given bellow is an induced subgraph of $G$ \begin{center}
			\begin{figure}[h]
				\includegraphics[height=8cm]{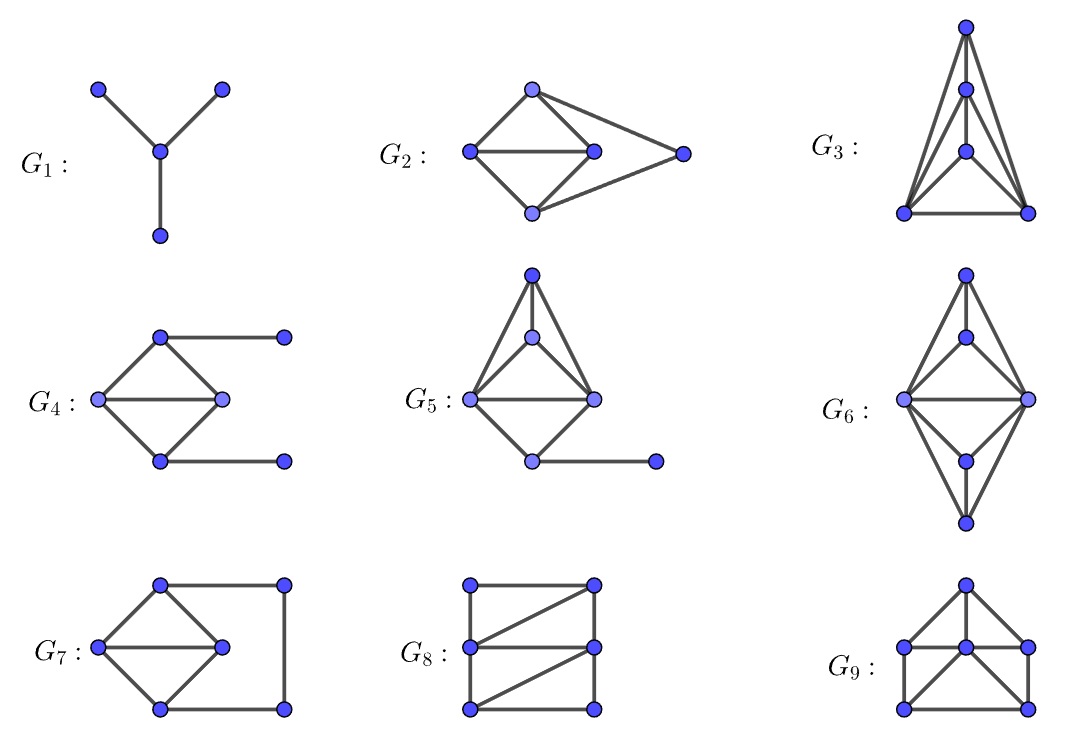}
				\caption{Forbidden induced subgraphs of the line graph of a graph}
			\end{figure}
		\end{center}
	\end{itemize}
\end{Theorem}

We can prove the following theorem which can be viewed as a ``generalization" of the implication (i)$\Rightarrow$(ii) in Theorem \ref{LG}.

\begin{Theorem}\label{charlinesc}Let $G$ be a finite simple graph and $\Delta$ a pure simplicial complex of dimension $d-1$.
If $G=\Lc(\Delta)$, then the edges of each connected component of $G$ can be partitioned into complete subgraphs in such a way that no vertex belongs to more than $d$ of the subgraphs. 
\end{Theorem}
\begin{proof}Without restricting the generality, we can assume that $G$ is connected. We construct a partition of the edges of $G$ into complete subgraphs, as follows: let $A_1$ be a maximal clique of $G$ with the largest number of vertices and we mark all the aedges from $A_1$. We choose a vertex of $A_1$ such that the set of its neighbours which are not in $A_1$ is maximal, and denote this vertex is $u_1$. Let $A_2$ be a maximal clique of $G$ which contains $u_1$, has the largest numbers of vertices and has the vertex set in $\{V(G)\setminus V(A_1)\}\cup\{u_1\}$. We mark all edges from $A_2$. If there are some others cliques which contains $u_1$ and the edges were not already been marked, we consider them now. Let's assume that there is no other clique containing $u_1$.  We continue to scan the vertices of $A_1$.  We assume that there is another vertex $u_2\in V(A_1)$ such that $\#(\mathcal{N}_G(u_2)\setminus V(A_1))$ is maximal. Let $A_3$ be a maximal clique of $G$ which contains $u_2$, has the largest number of vertices and $V(A_3)\subseteq\{V(G)\setminus V(A_1)\}\cup\{u_2\}$. Note that we allow all the edges which connect to the vertices already considered, but after each step, we mark the edges from the defined cliques. 
	We continue until we mark all the edges of $G$. Let's assume that for $E(G$) we obtained the partition $\{A_1,\ldots,A_s\}$, for some integer $s\geq1$.
	
	We have to prove that each vertex belongs to at most $d$ complete graphs $A_i$. By the construction and using Proposition \ref{degclique} the statement follows. Indeed, it is easy to see that Proposition \ref{degclique} says that any vertex belongs to at most $d$ maximal cliques (the number can be smaller). Moreover, we firstly we mark all the edges which are in the maximal cliques which contain $u_1$. Then we continue with the neighbours of $u_1$ and we do the same. But each maximal clique comes from a larger maximal clique which contains edge that were already marked in the above steps. So the number of maximal clique in the worst case remains the same.
\end{proof}
We exemplify the above construction:
\begin{Example}\rm\label{expart}
	Let $\Delta$ be the simplicial complex given by the set of all the $2$-faces of the simplicial complex $\Gamma=\langle\{x_1,x_2,x_3\}\rangle*\langle\{x_4\},\{x_5\},\{x_6\},\{x_7\}\rangle$. Hence, $$\Fc(\Delta)=\{F_1=\{x_1,x_2,x_3\},F_2=\{x_1,x_2,x_4\},F_3=\{x_1,x_3,x_4\},F_4=\{x_2,x_3,x_4\}$$
	$$F_5=\{x_1,x_2,x_5\},F_6=\{x_1,x_3,x_5\},F_7=\{x_2,x_3,x_5\},F_8=\{x_1,x_2,x_6\},$$ $$F_9=\{x_1,x_3,x_6\},F_{10}=\{x_2,x_3,x_6\}, F_{11}=\{x_1,x_2,x_7\},$$ $$F_{12}=\{x_1,x_3,x_7\},F_{13}=\{x_2,x_3,x_7\}\}.$$
	In $\Lc(\Delta)$ we have the following maximal cliques (we use $u_i$ for the facet $F_i$):
	
	\[\mbox{induced by} \langle\{x_1,x_2\}\rangle*\langle\{x_3\},\{x_4\},\{x_5\},\{x_6\},\{x_7\}\rangle:\ u_1,u_2,u_5,u_8,u_{11}\] 
	\[\mbox{induced by} \langle\{x_1,x_3\}\rangle*\langle\{x_2\},\{x_4\},\{x_5\},\{x_6\},\{x_7\}\rangle:\ u_1,u_3,u_6,u_9,u_{12}\] 
	\[\mbox{induced by} \langle\{x_2,x_3\}\rangle*\langle\{x_1\},\{x_4\},\{x_5\},\{x_6\},\{x_7\}\rangle:\ u_1,u_4,u_7,u_{10},u_{13}\] 
	\[\mbox{induced by the subsets of the set} \{x_1,x_2,x_3,x_4\}:\ u_1,u_2,u_3,u_4\] 
	\[\mbox{induced by the subsets of the set} \{x_1,x_2,x_3,x_5\}:\ u_1,u_5,u_6,u_7\]
	\[\mbox{induced by the subsets of the set} \{x_1,x_2,x_3,x_6\}:\ u_1,u_8,u_9,u_{10}\]
	\[\mbox{induced by the subsets of the set} \{x_1,x_2,x_3,x_7\}:\ u_1,u_{11},u_{12},u_{13}\]
	Note that there are not other edges.
	
	In order to construct the partition, we start with a maximal clique of largest size: $A_1=\{u_1,u_2,u_5,u_8,u_{11}\}$ and we mark all the edges from this clique. We have that $u_1$ is in some the maximal cliques, so we continue with it. Let $A_2=\{u_1,u_3,u_6,u_9,u_{12}\}$ and we mark all the edges. We continue with $u_1$ and we have $A_3=\{u_1,u_3,u_7,u_{10},u_{13}\}$ and we mark all the edges. Since all the edges from $u_1$ are marked, we continue with the vertices of $A_1$ and we choose $u_2$. Let $A_4=\{u_2,u_3,u_4\}$ and we mark all the edges from $A_4$. We continue with $u_5$ and we get $A_5=\{u_5,u_6,u_7\}$  and mark all the edges. We continue with $u_8$ and we get $A_6=\{u_8,u_9,u_{10}\}$ and mark all the edges. Finally, we continue with $u_{11}$ and we get $A_7=\{u_{11},u_{12},u_{13}\}$. Now all the edges are marked. 
	
	We see that $d=3$, $u_1$ is contained in three cliques: $A_1,A_2,A_3$ and all the other vertices are contained in two cliques.
	
	The figure bellow contains the line graph of $\Delta$ and the corresponding partition of edges.
	
	\begin{center}
		\begin{figure}[h]
			\includegraphics[height=6.5cm]{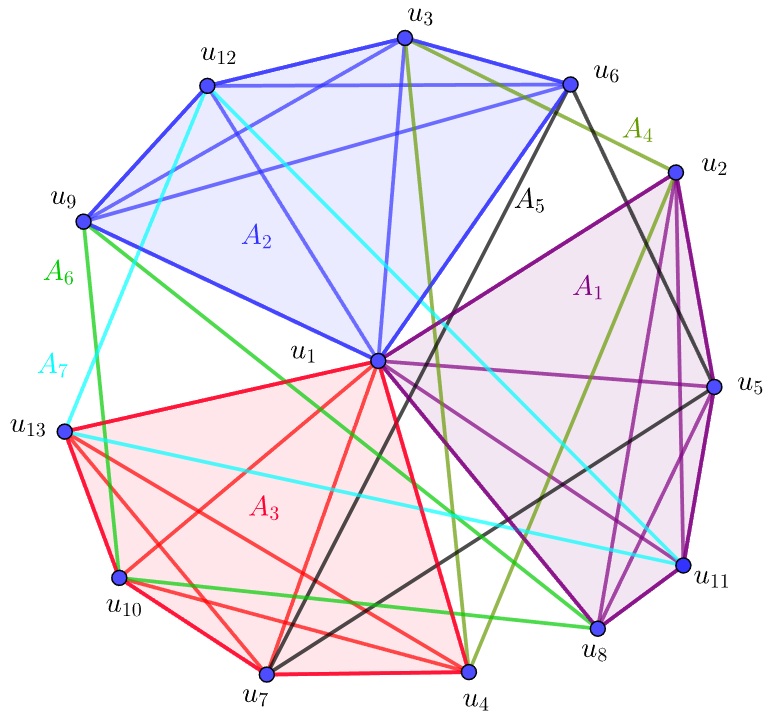}
			
		\end{figure}
	\end{center} 
\end{Example}
Note that the converse implication does not hold. More precisely, there are graphs $G$ for which the edges of each connected component of $G$ can be partitioned into complete subgraphs in such a way that no vertex belongs to more than $d$ of the subgraphs, but there is no simplicial complex $\Delta$ of dimension $d-1$ such that $\Lc(\Delta)=G$ as the following example shows.

\begin{Example}\rm\label{partitionK23}
	Let $G=\mathcal{K}_{2,3}$ with the edges $\{u_i,v_j\}$, $1\leq i\leq 2$, $1\leq j\leq 3$.
	\begin{center}
		\begin{figure}[h]
			\includegraphics[height=3cm]{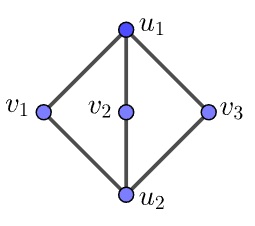}
			\caption{The complete bipartite graph $\mathcal{K}_{2,3}$}
		\end{figure}
	\end{center} 

Each edge form a complete graph. Thus $u_i$ is contained in $3$ complete subgraphs, $1\leq i\leq 2$ and $v_j$ is contained in two maximal cliques, $1\leq j\leq 3$, but there is not any pure simplicial complex $\Delta$ such that $\Lc(\Delta)=G$, by \cite[Remark 4.3]{BHS} or Proposition \ref{complete-bipartite-2n}. 
\end{Example}

Examples shows that there are many classes of finite simple graphs for which the converse implication is true. In the sequel we construct such an example:

\begin{Theorem}\label{triangle-wheel} Let $G$ be the friendship graph (see the figure bellow) with $2n+1$ vertices $\{v,u_1,\ldots,u_{2n}\}$, for some $n\geq 2$. Then there is a partition of the edges of $G$ into complete subgraphs $A_i=\{v,u_{2i-1},u_{2i}\}\}$, $1\leq i\leq n$ such that each vertex belongs to no more than $d=n$ complete subgraphs and $G$ is the line graph of a pure simplicial complex of dimension $d-1$.

\begin{center}
	\begin{figure}[h]
		\includegraphics[height=5cm]{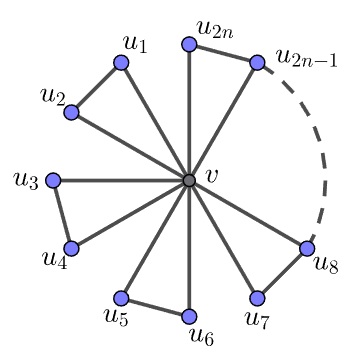}

\caption{The Friendship graph, $F_n$}	\end{figure}
\end{center}
	\end{Theorem}

\begin{proof}
	We will construct a simplicial complex $\Delta$ of dimension $d-1$ such that $\Lc(\Delta)=G$. Let's assume that to each complete subgraph $A_i$ one assigns a vertex $x_i$ of $\Delta$. Moreover, we consider as vertices of $\Delta$ all the vertices of $G$ which belong to exactly one $A_i$, that is all the vertices $u_1,\ldots,u_{2n}$. Let $F=\{x_1,\ldots,x_n\}$, $F_{2i-1}=\{u_{2i-1},x_1,\ldots,\widehat{x_i},\ldots,x_n\}$, and $F_{2i}=\{u_{2i},x_1,\ldots,\widehat{x_i},\ldots,x_n\}$. Then it is easy to see that the simplicial complex with the facet set $\Fc(\Delta)=\{F,F_1,\ldots,F_{2n}\}$ has dimension $d-1$ and $\Lc(\Delta)=G$.
\end{proof}

One may consider a slightly general result of the Theorem \ref{triangle-wheel} if one replaces all the triangles, which are in fact $\mathcal{K}_3$, with complete graphs or arbitrary dimensions (the generalized Dutch windmill graph). The proof is similar and we will skip it.

\begin{Theorem}
	Let $G$ be a generalized Dutch windmill graph, that is its edge set can be partitioned into maximal complete subgraphs $\mathcal{K}_{n_1},\ldots,\mathcal{K}_{n_d}$ such that $\bigcap\limits_{1\leq i\leq d}V(\mathcal{K}_{n_i})=\{v\}$. Then $G$ is the line graph of a pure simplicial complex of dimension $d-1$.
\end{Theorem}
	\section{Facet ideals and line graphs}
	
	For edge ideals of graphs one can describe the second Betti number in terms of the combinatorial structure of its line graph.
	\begin{Proposition}\cite[Proposition 2.1]{EV}\label{Betti}
		Let $I\subset S$ be the edge ideal of the graph $G$, let $V$ be the vertex set of $G$, and let $L(G)$ be the line graph of $G$. If \[\cdots\longrightarrow S^c(-4)\oplus S^{b}(-3)\longrightarrow S^q(-2)\longrightarrow S\longrightarrow S/I\longrightarrow0\]is the minimal graded resolution of $S/I$, then $b=\#(E(L(G)))-N_t$, where $N_t$ is the number of triangles of $G$ and $c$ is the number of unordered pairs of edges $\{f,g\}$ such that $f\cap g=\emptyset$ and $f$ and $g$ cannot be joined by an edge.
	\end{Proposition}
	We obtain a similar result for pure simplicial complexes by using the line graph.
	Let $\Delta$ be a pure simplicial complex of dimension $d-1$, $d\geq2$ on the vertex set $[n]=\{1,\ldots,n\}$, with the facet set $\mathcal{F}(\Delta)=\{F_1,\ldots,F_r\}$, $r\geq1$. Let $S=\kk[x_1,\ldots,x_n]$ be the polynomial ring in $n$ variables over a field $\kk$, and $I(\Delta)=(\mathbf{x}_{F_1}, \ldots,\mathbf{x}_{F_r})$ its facet ideal. 
	\begin{Theorem}\label{Bettisc}
		Let $I\subset S$ be the facet ideal of $\Delta$ and $\Lc(\Delta)$ its line graph. Let $N_t$ be the number of all the triangles in $\Lc(\Delta)$ which are disjoint (their vertex sets are disjoint) and don't arise from facets $F,G,H$ with $\#(F\cap G\cap H)=d-1$. Then $\beta_{2,d+1}(S/I)=\#(E(\Lc(\Delta)))-N_t$.
	\end{Theorem}
	\begin{proof}
		The proof is similar to \cite[Proposition 7.6.3]{Vi}. Let $F_1,\ldots,F_r$ be the facets of $\Delta$, and $I(\Delta)=(\mathbf{x}_{F_1}, \ldots,\mathbf{x}_{F_r})$ its facet ideal. Let $\psi: S^q(-d)\rightarrow S$, $\psi(e_i)=\mathbf{x}_{F_i}$, where $e_i$ is the $i$-th unit vector. Let $Z_1=\ker\psi$  and $Z_1'$ be the set of all the elements in $Z_1$ of degree $d+1$. We regard $\mathbf{x}_{F_i}$ as the vertices of $\Lc(\Delta)$.  Every edge $e=\{\mathbf{x}_{F_i},\mathbf{x}_{F_j}\}$ in $\Lc(\Delta)$ determines a syzygy $\syz(e)=ve_i-ue_j$, where $F_i=\{u\}\cup(F_i\cap F_j)$ and $F_j=\{v\}\cup(F_i\cap F_j)$ for some vertices $u,v\in V$. By \cite[Theorem 3.3.19]{Vi} the set of those syzygies generate $Z_1'$. Let $C_3=\{\mathbf{x}_{F_i},\mathbf{x}_{F_j},\mathbf{x}_{F_k}\}$ be a triangle in $\Lc(\Delta)$. One may note (or use Corollary \ref{C3}) that we must have either $\#(F_i\cap F_j\cap F_k)=d-1$ or $\#(F_i\cap F_j\cap F_k)=d-2$. 
		\begin{itemize}
			\item[$\bullet$] If $\#(F_i\cap F_j\cap F_k)=d-1$, then one must have $F_i=\{u\}\cup (F_i\cap F_j\cap F_k),$ $F_j=\{v\}\cup (F_i\cap F_j\cap F_k),$ and $F_k=\{w\}\cup (F_i\cap F_j\cap F_k)$ for some vertices $u,v,w$. Let $$\phi(C_3)=\{ve_i-ue_j,\ we_j-ve_k,\ we_i-ue_k\}.$$ It is easy to see that all the elements from this set are linearly independent.

		\item[$\bullet$] If $\#(F_i\cap F_j\cap F_k)=d-2$, one must have $F_i=\{u,v\}\cup (F_i\cap F_j\cap F_k),$ $F_j=\{v,w\}\cup (F_i\cap F_j\cap F_k),$ and $F_k=\{u,w\}\cup (F_i\cap F_j\cap F_k)$ for some vertices $u,v,w$. Let $\phi(C_3)=\{we_i-ue_j,\ ue_j-ve_k,\ we_i-ve_k\}.$
		In this case $we_i-ve_k=we_i-ue_j+ue_j-ve_k,$ hence the elements of $\phi(C_3)$  are linearly dependent.
	\end{itemize}
		One may note that if we choose two disjoint triangles $C_3$ and $C_3'$ of $\Lc(\Delta)$, we get $\phi(C_3)\cap\phi(C_3')=\emptyset$. Let $T$ be the set of all the triangles in $\Lc(\Delta)$ which are disjoint and don't arise from facets $F,G,H$ with $\#(F\cap G\cap H)=d-1$. From every triangle in $T$, choose an element $\rho(C_3)\in\phi(C_3)$. Then \[B=\{\syz(e)|e\in E(\Lc(\Delta))\}\setminus\{\rho(C_3):C_3\in T\}\]is a minimal generating set for $Z_1'$. The statement follows.
	\end{proof}
	\begin{Remark}\rm
		Note that the formula obtained does not depend on the characteristic of the ground field since the second Betti number of a Stanley--Reisner is independent of the ground field \cite{HT}.
	\end{Remark}
	\begin{Problem}\rm We cannot obtain a similar result for the other graded Betti numbers $\beta_{2,i}(S/I)$ in terms of $\Lc(\Delta)$, but examples shows that their description is encoded in the combinatorics of the other $k$-line graphs, with $k<d-1$. Therefore, taking into account Proposition~\ref{Betti}, is there a similar formula for $\beta_{2,d+2}(S/I)$ in terms of the combinatorics of the $k$-line graphs?
	\end{Problem}
	Note that we can translate the above result in terms of the Alexander duality:
	\begin{Corollary}\label{Bettiscaldual}
		Let $\Delta$ be a pure simplicial complex of dimension $d-1$ with $n>d+1$ vertices, $I\subset S=\kk[x_1,\ldots,x_n]$ the facet ideal of $\Delta$, and $\Lc(\Delta)$ its line graph. Let $N_t$ be the number of all the triangles in $\Lc(\Delta)$ which are disjoint (their vertex sets are disjoint) and don't arise from facets $F,G,H$ with $\#(F\cap G\cap H)=d-1$. Then $\beta_{2,n-d+1}(S/I_{\Delta^{\vee}})=\#(E(\Lc(\Delta)))-N_t$.
	\end{Corollary}
	\begin{proof}
		The proof is straightforward since Proposition \ref{Deltac} says that $\Lc(\Delta)=\Lc(\Delta^c)$ and Proposition \ref{Aldualcomp} says that $I_{\Delta^{\vee}}=I(\Delta^c)$. The statement follows by Theorem~\ref{Bettisc}.
	\end{proof}
	
	\section{Final conclusions and remarks}
	
	One may note that, in the case of line graphs of a simplicial complex, a characterization in terms of a finite list of forbidden induced subgraphs does not exist. (see Propositions \ref{wheel} and \ref{k22} for instance). The following problem naturally arise:
	
	\begin{Problem}\rm
		Are there some other classes of forbidden induced subgraphs for line graphs of simplicial complexes? 
	\end{Problem}

This automatically imply that a characterization of those graphs which are line graphs of some pure simplicial complexes as the one from Theorem \ref{LG} ``(i)$\Leftrightarrow$(iv)" is not possible in the same form. Still, the problem of generalising Theorem \ref{LG} can be seen from a different point of view, since in the case of line graphs of graphs, the simplicial complex has dimension one. For instance:

\begin{Problem}\rm
	Let $G$ be a finite simple graphs, $d\geq2$ an integer and $\Delta$ a pure $(d-1)$-simplicial complex such that $\Lc(\Delta)=G$. Is there a forbidden (finite) list of (classes) of induced subgraphs of $G$?
\end{Problem}

We expect that the forbidden classes of graphs depends on the fixed integer $d\geq 2$ (as seen in  Propositions \ref{wheel}(i) and \ref{1,d+1}). In particular, taking into account the approach from \cite[pp 159]{NRSS} and Theorem \ref{G1t}, the following problem is of interest:

\begin{Problem}\rm
	Let $G$ be a finite simple graphs, and $\Delta$ a pure $2$-simplicial complex such that $\Lc(\Delta)=G$. Is there a forbidden (finite) list of (classes) of induced subgraphs of $G$?
\end{Problem}

We saw in Theorem \ref{charlinesc} that an implication similar to Theorem \ref{LG} ``(i)$\Rightarrow$(ii)" holds, but it is not necessary. Nevertheless, there are classes of graphs for which the condition is sufficient. Examples shows that chordal graphs also have this property.

\begin{Problem}\rm
	Let $G$ be a finite simple chordal graph. Is there a partition of the edges of $G$ into complete subgraphs $A_1,\ldots,A_r $ such that each vertex belongs to no more than $d=r$ complete subgraphs and $G$ is the line graph of a pure simplicial complex of dimension $d-1$?
\end{Problem}

The next example shows such a construction.

	\begin{Example}\rm\label{Chordal}
		Let $G$ be the graph 
		
		\begin{figure}[h]
			\includegraphics[height=4cm]{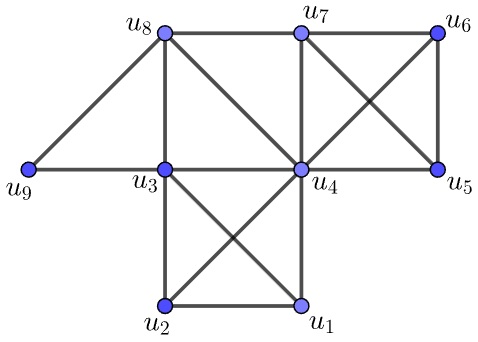}
		\end{figure}
		
		In order to construct the simplicial complex, we consider a partition of $E(G)$ as follows:  $A_1=\{u_1,u_2,u_3,u_4\}$, $A_2=\{u_4,u_5,u_6,u_7\}$, $A_3=\{u_3,u_8,u_9\}$, $A_4=\{u_4,u_8\}$ and $A_5=\{u_7,u_8\}$. Therefore, for $E(G)$ we obtained the partition $\{A_1,A_2,A_3,A_4,A_5\}$
		
	 We determine the dimension of the simplicial complex $\Delta$. Firstly, we define the following integers: $$d_1=\max\{\#(V(A_i)):1\leq i\leq 4\},$$ $$d_2=\max\{t\ : V(A_1)\cap\cdots\cap V(A_t)\neq\emptyset\},$$and $$d_3=\diam(G).$$
	 Let $d=\max\{d_1,d_2,d_3\}$ and $d-1$ be the dimension of $\Delta$.
		Note that $d_1=4$, $d_2=3$ since $A_1\cap A_2\cap A_4\neq \emptyset$, and $d_3=3$. Thus $d=\max\{d_1,d_2,d_3\}=4$, so $\Delta$ will be a 3-dimensional simplicial complex.

	 We construct the facets of $\Delta$, by using Theorem \ref{complete}. Note that each facet has $4$ vertices. We start by constructing the facets which correspond to the vertices of $A_1$. We consider \[\Delta_1=\langle\{F_1,F_2,F_3,F_4\}\rangle.\] where the facets are subsets of cardinality 4 of the 4-simplex $\langle\{x_1,x_2,x_3,x_4,x_5\}\rangle$: $F_1=\{x_1,x_2,x_3,x_4\}$, $F_2=\{x_1,x_2,x_3,x_5\}$, $F_3=\{x_1,x_3,x_4,x_5\}$, and $F_4=\{x_1,x_2,x_4,x_5\}$ which correspond to $u_1$, $u_2$, $u_3$, and $u_4$ respectively.
		
		Since $V(A_1)\cap V(A_2)=\{u_4\}$, we consider\[\Delta_2=\langle\{x_1\},\{y_1\},\{y_2\},\{y_3\}\rangle*\langle\{x_2,x_4,x_5\}\rangle.\] So the facets which correspond to $u_5$, $u_6$ and $u_7$ are $F_5=\{y_1,x_2,x_4,x_5\}$, $F_6=\{y_2,x_2,x_4,x_5\}$, and $F_7=\{y_3,x_2,x_4,x_5\}$.
		
		We continue with $u_4$, since $u_4\in V(A_4)$. On the other hand, $u_3,u_4,u_8$ form a cycle of length 3. Moreover  $u_4,u_7,u_8$ form a cycle of length 3, So $$\#(F_8\cap F_3)=\#(F_8\cap F_4)=\#(F_8\cap F_7)=3.$$ Since $F_3\cup F_4\cup F_7=\{x_1,x_2,x_3,x_4,x_5,y_3\}$, we consider $F_8=\{x_1,x_4,x_5,y_3\}$. Note that also $\{u_7,u_8\}$ is an edge, so we should have $\#(F_7\cap F_8)=3$ which is true since $F_7\cap F_8=\{y_3,x_4,x_5\}$.
		
		Therefore, we also obtained $\Delta_4=\langle F_4,F_8\rangle$ and
		$\Delta_5=\langle F_7,F_8\rangle$.
		
		For the $A_3$, we observe that $u_3,u_8,u_9$ form a cycle of length 3 and $ F_3\cup F_8=\{x_3,x_1,x_4,x_5,y_3\}$. So we choose $F_9=\{x_3,x_4,x_5,y_3\}$ (by using Corollary~\ref{C3} b).
		
		We obtained the simplicial complex $\Delta=\langle F_1,\ldots,F_9\rangle$ on the set of vertices $V=\{x_1,x_2,x_3,x_4,x_5,y_1,y_2,y_3\}$ which has the property that $\Lc(\Delta)=G$.
	\end{Example}

\section*{Acknowledgements}  
The author thanks Margaret Bayer for pointing out that line graphs of simplicial complexes have been studied in several other places as facet-ridge incident graphs. The author is grateful to the anonymous reviewer who proposed the graphs from Examples \ref{expart}  and \ref{Chordal} and made several suggestions.
	
\end{document}